\documentclass[12pt]{amsart}
\usepackage{a4wide,enumerate}
\allowdisplaybreaks

\let\pa\partial  
\let\na\nabla  
\let\eps\varepsilon  
  
\newcommand{\R}{{\mathbb R}} 
\newcommand{\diver}{\operatorname{div}}

\newtheorem{theorem}{Theorem}   
\newtheorem{lemma}[theorem]{Lemma}   
   
\newtheorem{remark}[theorem]{Remark}   
\newtheorem{corollary}[theorem]{Corollary}

\begin{document}  

\title[Boundedness of weak solutions]{Boundedness of weak solutions
to cross-diffusion systems from population dynamics}

\author{Ansgar J\"ungel}
\address{Institute for Analysis and Scientific Computing, Vienna University of  
	Technology, Wiedner Hauptstra\ss e 8--10, 1040 Wien, Austria}
\email{juengel@tuwien.ac.at} 

\author{Nicola Zamponi}
\address{Institute for Analysis and Scientific Computing, Vienna University of  
	Technology, Wiedner Hauptstra\ss e 8--10, 1040 Wien, Austria}
\email{nicola.zamponi@tuwien.ac.at}

\date{\today}

\thanks{The authors acknowledge partial support from   
the Austrian Science Fund (FWF), grants P22108, P24304, and W1245, and    
the Austrian-French Program of the Austrian Exchange Service (\"OAD)} 

\begin{abstract}
The global-in-time existence of nonnegative bounded weak solutions 
to a class of cross-diffusion systems for two population species is proved. 
The diffusivities are assumed to depend linearly on the population 
densities in such a way that a certain formal gradient-flow structure holds.
The main feature of these systems is 
that the diffusion matrix may be neither symmetric nor positive definite.
The key idea of the proof is to employ the boundedness-by-entropy principle 
which yields at the same time the existence of global weak solutions and 
their boundedness. In particular, the uniform boundedness of weak solutions
to the population model of Shigesada, Kawasaki, and Teramoto in several space 
dimensions under certain conditions on the diffusivities is shown
for the first time.
\end{abstract}

 

\maketitle


\section{Introduction}\label{sec.intro}

Many multi-species systems in biology, chemistry, and physics can be described
by reaction-diffusion systems with cross-diffusion effects. The independent 
variables usually represent densities or concentrations of the species and thus,
they should be nonnegative and bounded. 
However, the proof of these properties is a challenging problem since maximum 
principle arguments generally cannot be applied to such systems. In fact, it is 
well known that weak solutions may be unbounded \cite{StJo95}. 
A further mathematical
difficulty comes from the fact that many systems have diffusion matrices which
are neither symmetric nor positive definite such that even the local-in-time
existence of solutions may be a nontrivial task.

Recently, a systematic method has been developed to overcome these difficulties for 
cross-diffusion systems which possess a formal gradient-flow structure
\cite{Jue14}. The so-called bound\-ed\-ness-by-entropy principle allows us, under
certain assumptions, to prove the existence of global-in-time 
nonnegative bounded weak solutions.
In this note, we determine a class of cross-diffusion systems, whose diffusivities
depend linearly on the solution and for which global
bounded weak solutions exist. In particular, we prove for the first time
the uniform boundedness of weak solutions to a class of population systems of 
Shigesada-Kawasaki-Teramoto type in several space dimensions \cite{SKT79}. 

More specifically, we consider reaction-diffusion systems of the form
\begin{equation}\label{1.eq}
  \pa_t u - \diver(A(u)\na u) = f(u)\quad\mbox{in }\Omega,\ t>0,
\end{equation}
subject to the homogeneous Neumann boundary and initial conditions
\begin{equation}\label{1.bic}
	(A(u)\na u)\cdot\nu=0\quad\mbox{on }\pa\Omega, \quad u(0)=u^0\quad\mbox{in }\Omega,
\end{equation}
where $u=(u_1,u_2)^\top$ represents the vector of the
densities or concentrations of the species, 
$A(u)=(A_{ij}(u))\in\R^{2\times 2}$ is the diffusion matrix,
and the reactions are modeled by the function $f=(f_1,f_2)$.
Furthermore, $\Omega\subset\R^d$ ($d\ge 1$) is a bounded domain
with Lipschitz boundary and $\nu$ is the exterior unit normal vector to
$\pa\Omega$. Our main assumption is that the diffusivities
depend linearly on the densities,
\begin{equation}\label{1.Aij}
  A_{ij}(u) = \alpha_{ij} + \beta_{ij}u_1 + \gamma_{ij}u_2. \quad i,j=1,2,
\end{equation}
where $\alpha_{ij}$, $\beta_{ij}$, $\gamma_{ij}$ are real numbers.

Such models can be formally derived from a master equation for a random walk
on a lattice in the diffusion limit with transition rates which depend
linearly on the species' densities \cite[Appendix B]{Jue14}. 
They can be also deduced as the limit equations of an interacting particle system
modeled by stochastic differential equations with interaction forces
which depend linearly on the corresponding stochastic processes
\cite{GaSe14,Oel89}. 

Cross-diffusion systems with linear diffusivities arise, for instance,
in the modeling of ion transport through narrow channels \cite{BDPS10,BSW12},
in population dynamics with complete segregation \cite{BuTr83}, and in the
population model of \cite{Dav14} with the diffusion matrix
$$
  A(u) = \begin{pmatrix} 1-u_1 & -u_1 \\ -u_2 & 1-u_2 \end{pmatrix}.
$$
The most prominent example is probably the 
population model of Shigesada, Kawasaki, and Teramoto \cite{SKT79}:
\begin{equation}\label{1.SKT}
  A(u) = \begin{pmatrix}
	a_{10} + 2a_{11}u_1 + a_{12}u_2 & a_{12}u_1 \\
	a_{21}u_2 & a_{20} + a_{21}u_1 + 2a_{22}u_2
	\end{pmatrix},
\end{equation}
where the coefficients $a_{ij}$ are nonnegative.
The numbers $a_{11}$, $a_{22}$ are called self-diffusion coefficients, and 
$a_{12}$, $a_{21}$ are referred to as cross-diffusion constants.
In this model, the source terms in \eqref{1.eq} are given by
\begin{equation}\label{1.LV}
  f_i(u) = (b_{i0} - b_{i1}u_1 - b_{i2}u_2)u_i, \quad i=1,2,
\end{equation}
where the coefficients $b_{11}$, $b_{22}$ are the intra-specific competition
constants and $b_{12}$, $b_{21}$ the inter-specific competition coefficients. 
The existence of global weak solutions without any restriction on the
diffusivities (except positivity) was achieved in 
\cite{GGJ03} in one space dimension and in \cite{ChJu04,ChJu06} 
in several space dimensions. Existence results for
nonlinear coefficients $A_{ij}(u)$ were proved in \cite{DLM13,Jue14}.
Less results are known concerning $L^\infty$ bounds.
In one space dimension and with coefficients $a_{10}=a_{20}$, 
Shim \cite{Shi02} proved uniform upper bounds. Moreover, if cross-diffusion 
is weaker than self-diffusion (i.e.\ $a_{12}<a_{22}$, 
$a_{21}<a_{11}$), weak solutions are bounded
and H\"older continuous \cite{Le06}. The existence of global bounded solutions
in the triangular case (i.e.\ $a_{21}=0$) was shown in \cite{CLY03}.
In this note, we prove uniform $L^\infty$ bounds under more general conditions 
on the coefficients $A_{ij}$ than in previous works.

The proof of global existence and boundedness of weak solutions 
to \eqref{1.eq}-\eqref{1.bic} is based on
the bound\-ed\-ness-by-entropy principle presented in \cite{Jue14}.
The key idea is to exploit the formal gradient-flow structure of \eqref{1.eq},
\begin{equation}\label{1.gf}
  \pa_t u - \diver\left(B\na\frac{\delta{\mathcal H}}{\delta u}\right) = f(u),
\end{equation}
where $B$ is a positive semidefinite matrix, $\delta{\mathcal H}/\delta u$
is the variational derivative of the entropy ${\mathcal H }[u]=\int_\Omega h(u)dx$,
and $h$ is the entropy density, which is assumed to be a function from $D\subset\R^2$ 
to $[0,\infty)$. Identifying $\delta{\mathcal H}/\delta u$ with its Riesz
representative $Dh(u)$ (the derivative of $h$) and introducing the entropy
variable $w=Dh(u)$, the above system can be formulated as
$$
  \pa_t u - \diver(B(w)\na w) = f(u),
$$
where $B=B(w)=A(u)(D^2h(u))^{-1}$, $D^2h$ is the Hessian of $h$, and
$u=(Dh)^{-1}(w)$. This formulation makes only sense if 
$Dh:D\to\R^2$ is invertible. 

There are two consequences of this formulation. First, if $f(u)\cdot w\le 0$,
the entropy ${\mathcal H}$ is a Lyapunov functional along solutions to 
\eqref{1.eq}-\eqref{1.bic} since
$$
  \frac{d{\mathcal H}}{dt} = \int_\Omega \pa_t u\cdot wdx
	\leq -\int_\Omega \na w:B(w)\na w dx = -\int_\Omega\na u:(D^2h)A(u)\na u dx\le 0,
$$
taking into account that  $B(w)$ (or equivalently $(D^2h)A(u)$) is assumed 
to be positive semidefinite.
Here, ``:'' denotes the double-dot product with summation over both matrix indices. 
Second, because of the invertibility of $Dh$, the original variable satisfies 
$u(x,t)=(Dh)^{-1}(w(x,t))\in D$, and if $D$ is a bounded domain, we obtain
automatically $L^\infty$ bounds without the use of a maximum principle.

In this note, we define the domain $D$ as the triangle
\begin{equation}\label{1.D}
  D = \{(u_1,u_2)\in\R^2:u_1>0,\ u_2>0,\ u_1+u_2<1\}.
\end{equation}
Then our main result is as follows. 

\begin{theorem}[Bounded weak solutions to \eqref{1.eq}]\label{thm.ex}
Let $u^0=(u_1^0,u_2^0)\in L^1(\Omega;\R^2)$ be such that $u^0(x)\in D$ 
for $x\in\Omega$, let $A(u)$ be given
by \eqref{1.Aij} with coefficients satisfying
\begin{align}
  & \alpha_{12} = \alpha_{21} = \beta_{21} = \gamma_{12} = 0, \label{1.symm1} \\
	& \beta_{22} = \beta_{11}-\gamma_{21}, \quad
	\gamma_{11} = \gamma_{22}-\beta_{12}, \quad 
	\gamma_{21} = \alpha_{22} - \alpha_{11} + \beta_{12}, \label{1.symm2} \\
	& \alpha_{11}>0,\quad \alpha_{22}>0, \quad 
	\beta_{12}<\alpha_{11}+\min\{\beta_{11},\gamma_{22}\}, \quad
	\alpha_{11}+\beta_{11} \ge 0, \quad \alpha_{22}+\gamma_{22} \ge 0, \label{1.cond}
\end{align}
and let $f_i(u)=u_ig_i(u)$, where $g_i(u)$ is continuous in $\overline{D}$
and nonpositive in $\{1-\eps < u_1+u_2 < 1\}$ for some $\eps>0$ $(i=1,2)$.
Then there exists a bounded nonnegative weak solution $u=(u_1,u_2)$ to 
\eqref{1.eq}-\eqref{1.bic} satisfying $u(x,t)\in \overline{D}$ for $x\in\Omega$, $t>0$,
\begin{equation}\label{1.reg}
  u\in L^2_{\rm loc}(0,\infty;H^1(\Omega;\R^2)), \quad
	\pa_t u\in L^2_{\rm loc}(0,\infty;H^1(\Omega;\R^2)'),
\end{equation}
and the initial datum is satisfied in the sense of $H^1(\Omega;\R^2)'$.
\end{theorem}

Note that the $L^\infty$ bound on $u$ is uniform in time.
The theorem also holds true if $\alpha_{11}=\alpha_{22}=0$ but
$\beta_{11}>0$ and $\gamma_{22}>0$; see Remark \ref{rem}.
The condition $u_1^0+u_2^0<1$ can be satisfied after a suitable scaling
of the positive function $u^0\in L^\infty(\Omega;\R^2)$ and is therefore not a 
restriction. The assumption on $f(u)$ guarantees that
the triangle $D$ is an invariant region under the action of the reaction terms.
If the population approaches its total capacity $u_1+u_2=1$,
the reaction terms are nonpositive and lead to a decrease of the population. 
Theorem \ref{thm.ex} generalizes the global existence result in \cite{GaSe14},
where the positive definiteness of $A$ was needed. To the best of our knowledge,
this is the first general existence result for uniformly bounded weak solutions
to cross-diffusion systems with linear diffusivities.

Choosing the diffusion matrix as in the population model \eqref{1.SKT},
we obtain the following corollary.

\begin{corollary}[Bounded weak solutions to \eqref{1.SKT}]\label{coro.ex}
Let the assumptions of Theorem \ref{thm.ex} hold except that the coefficients
of $A$, defined in \eqref{1.SKT}, are nonnegative and satisfy
$a_{10}>0$, $a_{20}>0$ as well as
\begin{equation}\label{1.a}
  a_{21} = a_{11}, \quad a_{22} = a_{12}, \quad a_{20}-a_{10} = a_{11}-a_{22}\ge 0.
\end{equation}
Furthermore, let $f(u)$ be given by the Lotka-Volterra terms \eqref{1.LV}
satisfying
\begin{equation}\label{1.b}
  b_{10} \le \min\{b_{11},b_{12}\}, \quad b_{20} \le \min\{b_{21},b_{22}\}.
\end{equation}
Then there exists a bounded weak solution $u=(u_1,u_2)$ to \eqref{1.eq}-\eqref{1.bic}
satisfying $u_1$, $u_2\ge 0$, $u_1+u_2\le 1$ in $\Omega\times(0,\infty)$, and
\eqref{1.reg}.
\end{corollary}

The novelty of this corollary is not the global existence result
(which has been already proven in \cite{ChJu04})
but the uniform boundedness of weak solutions. 
By fixing the numbering of the species, we may assume without loss of generality 
that $a_{20}\ge a_{10}$ (see \eqref{1.a}).
With conditions \eqref{1.a}, the diffusion matrix of the population model becomes
$$
  A(u) = \begin{pmatrix}
	a_{10} + 2a_{11}u_1 + a_{12}u_2 & a_{12}u_1 \\
	a_{11}u_2 & a_{20} + a_{11}u_1 + 2a_{12}u_2
	\end{pmatrix}, \quad\mbox{where }a_{12} = a_{11}+a_{10}-a_{20},
$$
i.e., we are left with three parameters $a_{10}$, $a_{20}$, and $a_{11}$.
Conditions \eqref{1.a} and \eqref{1.b} can be interpreted as follows.
The cross-diffusion coefficient of one species is the same
as the self-diffusion of the other species. Moreover, 
the self-diffusion of species no.\ 1 is larger than that for species no.\ 2 since
$a_{11}-a_{22}=a_{20}-a_{10}\ge 0$.
Condition \eqref{1.b} means that the growth rates $b_{10}$, $b_{20}$ 
are assumed to be not larger than the intra- and inter-specific competition rates.

The population model with \eqref{1.SKT} and \eqref{1.LV} possesses the natural
entropy structure \eqref{1.gf} with $h(u)=a_{12}^{-1}u_1(\log u_1-1)
+ a_{21}^{-1}u_2(\log u_2-1)$;
see \cite{ChJu04} for details. In particular, $(D^2h)A(u)$ is positive semidefinite.
However, its derivative $Dh(u)=(\log u_1,\log u_2)$
is defined on $(0,\infty)^2$, thus not yielding $L^\infty$ bounds.
We propose to employ the entropy density
\begin{equation}\label{1.h}
  h(u) = u_1(\log u_1-1) + u_2(\log u_2-1) + (1-u_1-u_2)(\log(1-u_1-u_2)-1),
\end{equation}
defined on the triangle \eqref{1.D}. The additional term gives the desired bounds
since $Dh(u)=(\log(u_1/(1-u_1-u_2)),\log(u_2/(1-u_1-u_2)))$ is defined
on $D$ which is bounded. However, in order to guarantee the positive semidefiniteness
of $(D^2h)A(u)$ (in fact, we need a slightly stronger property; 
see Section \ref{sec.thm}), we impose conditions \eqref{1.a} and
\eqref{1.b} which are not needed for the global existence analysis. 
We conjecture that these conditions are not necessary 
to prove the boundedness of weak solutions but this is an open problem.

The key idea of the proof of Theorem \ref{thm.ex} is to apply the
general existence result from \cite[Theorem 2]{Jue14}.
Hence, we just need to verify the hypotheses of this theorem.
One of these assumptions states that the matrix $(D^2h)A(u)$ has to be positive
semidefinite (we need a slightly stronger property). 
Although this is only a $2\times 2$ matrix, the proof is not
trivial because we have to deal with twelve parameters $\alpha_{ij}$, $\beta_{ij}$,
and $\gamma_{ij}$. In order to reduce the complexity of the problem, we assume
that $B(w)$ or, equivalently, $(D^2h)A(u)$ is symmetric, motivated by
the Onsager symmetry in non-equilibrium thermodynamics. 
This yields seven conditions, 
and we are left with five parameters. By Sylvester's criterion,
the positive semidefiniteness follows if one of the diagonal terms and the
determinant of $(D^2h)A(u)$ are nonnegative. The corresponding expressions
are quadratic polynomials in $(u_1,u_2)$. The task of finding conditions
on the parameters such that these polynomials are nonnegative can be solved
in principle by Cylindrical Algebraic Decomposition \cite{CaJo98}. 
This yields a complicated set of conditions. Therefore, we prefer another
approach based on the strong maximum principle applied to 
multivariate quadratic polynomials, which leads to \eqref{1.a}.
We stress the fact that the maximum principle
is {\em not} needed to prove the $L^\infty$ bounds but to solve the
algebraic problem.

This note is organized as follows.
Theorem \ref{thm.ex} is proved in Section \ref{sec.thm}, whereas
in Section \ref{sec.coro}, Corollary \ref{coro.ex} is shown.


\section{Proof of Theorem \ref{thm.ex}}\label{sec.thm}

We apply the following theorem from \cite[Theorem 2]{Jue14}, here in a formulation
which is adapted to our situation.

\begin{theorem}[\cite{Jue14}]\label{thm.Jue14}
Let $D\subset(0,1)^2$ be a bounded domain, $u^0\in L^1(\Omega;\R^2)$ with 
$u^0(x)\in D$ for $x\in\Omega$ and assume that
\begin{description}
\item[\rm H1] There exists a convex function $h\in C^2(D;[0,\infty))$ 
such that its derivative $Dh:D\to\R^n$ is invertible.
\item[\rm H2] Let $\alpha^*>0$, $0\le m_i\le 1$ $(i=1,2)$ be such that for all 
$z=(z_1,z_2)^\top\in\R^2$ and $u=(u_1,u_2)^\top\in D$,
$$
  z^\top D^2h(u)A(u)z \ge \alpha^*\sum_{i=1}^2 u_i^{2(m_i-1)}z_i^2.
$$
\item[\rm H3] It holds $A\in C^0(D;\R^{2\times 2})$ and there exists $c_f>0$ 
such that for all $u\in D$, $f(u)\cdot Dh(u)\le c_f(1+h(u))$.
\end{description}
Then there exists a weak solution $u$ to \eqref{1.eq}-\eqref{1.bic}
satisfying $u(x,t)\in\overline{D}$ for $x\in\Omega$, $t>0$ and
\begin{equation*}
  u\in L^2_{\rm loc}(0,\infty;H^1(\Omega;\R^2)), 
	\quad\pa_t u \in L^2_{\rm loc}(0,\infty;H^1(\Omega;\R^2)').
\end{equation*}
The initial datum is satisfied in the sense of $H^1(\Omega;\R^2)'$.
\end{theorem}

Choosing the entropy function \eqref{1.h}
defined on $D$ (see \eqref{1.D}), we see that Hypothesis H1 is satisfied.
It remains to verify Hypotheses H2 and H3.

\subsection{Verification of Hypothesis H2}

Let $H(u)=(D^2h)(u)$.
As explained in the introduction, 
we require that the matrix $H(u)A(u)$ is symmetric. 
This leads to conditions \eqref{1.symm1}-\eqref{1.symm2},
and we are left with the five parameters $\alpha_{11}$, $\alpha_{22}$, $\beta_{11}$,
$\beta_{12}$, and $\gamma_{22}$. We prove that $H(u)A(u)$ is positive definite
under additional assumptions.

\begin{lemma}\label{lem.HA}
Let conditions \eqref{1.symm1}-\eqref{1.cond} hold.
Then there exists $\eps>0$ such that for all $z\in\R^2$ and all $u\in D$,
\begin{equation}\label{a.H2}
  z^\top H(u)A(u)z \ge \eps\left(\frac{z_1^2}{u_1}+\frac{z_2^2}{u_2}\right).
\end{equation}
\end{lemma}

The lemma shows that Hypothesis H2 is fulfilled with $m_i=\frac12$.
First, we show the following result.

\begin{lemma}\label{lem.HA2}
The matrix $H(u)A(u)$ is positive semidefinite for all $u\in D$ if and only if
\begin{equation}\label{a.cond}
  \alpha_{11}\ge 0,\quad \alpha_{22}\ge 0, \quad 
	\beta_{12}\le\alpha_{11}+\min\{\beta_{11},\gamma_{22}\}, \quad
	\alpha_{11}+\beta_{11}\ge 0, \quad \alpha_{22}+\gamma_{22}\ge 0.
\end{equation}
\end{lemma}

\begin{proof}
{\em Step 1: Equations \eqref{a.cond} are necessary.}
We first prove that the positive semidefiniteness of $H(u)A(u)$ implies
\eqref{a.cond} by studying $H(u)A(u)$ close to the vertices of $D$. To this end,
we define the matrix-valued functions
\begin{align*}
  & F_1(s)=s H(s,s)A(s,s), \quad F_2(s) = sH(1-2s,s)A(1-2s,s), \\
	& F_3(s) = sH(s,1-2s)A(s,1-2s) \quad\mbox{for } s\in(0,\tfrac12).
\end{align*}
A straightforward computation shows that
\begin{align*}
  \lim_{s\to 0+}F_1(s) &= \begin{pmatrix} \alpha_{11} & 0 \\ 0 & \alpha_{22}
	\end{pmatrix}, \quad
  \lim_{s\to 0+}F_2(s) = \begin{pmatrix} 
	\alpha_{11}+\beta_{11} & \alpha_{11}+\beta_{11} \\
	\alpha_{11}+\beta_{11} & 2(\alpha_{11}+\beta_{11})-\beta_{12} \end{pmatrix}, \\
	\lim_{s\to 0+}F_3(s) &= \begin{pmatrix} 
	\alpha_{11}+\alpha_{22}+2\gamma_{22}-\beta_{12} & \alpha_{22}+\gamma_{22} \\
	\alpha_{22}+\gamma_{22} & \alpha_{22}+\gamma_{22}
	\end{pmatrix}.
\end{align*}
Since $H(u)A(u)$ is assumed to be positive semidefinite on $D$, also
$\lim_{s\to 0+}F_i(s)$ must be positive semidefinite for $i=1,2,3$.
Sylvester's criterion applied to these matrices yields \eqref{a.cond} since
\begin{align*}
  \det\big(\lim_{s\to 0+}F_2(s)\big) &= (\alpha_{11}+\beta_{11})
	(\alpha_{11}+\beta_{11}-\beta_{12}) \ge 0, \\
  \det\big(\lim_{s\to 0+}F_3(s)\big) &= (\alpha_{22}+\gamma_{22})
	(\alpha_{11}+\gamma_{22}-\beta_{12}) \ge 0.
\end{align*}

{\em Step 2: Sign of the diagonal elements of $HA$.}
Let conditions \eqref{a.cond} hold.
We claim that either $HA=H(u)A(u)$ is positive semidefinite or one of the
two coefficients $(HA)_{11}$ or $(HA)_{22}$ is positive in $D$.
For this, we introduce the functions
\begin{align*}
  f_1(u_2,u_3) &= (1-u_2-u_3)u_3(HA)_{11}(1-u_2-u_3,u_2), \quad (u_2,u_3)\in D, \\
	f_2(u_1,u_3) &= (1-u_1-u_3)u_3(HA)_{22}(u_1,1-u_1-u_3), \quad (u_1,u_3)\in D.
\end{align*}
We wish to apply the strong maximum principle to $f_1$ and $f_2$. In fact,
$f_1$ and $f_2$ are nonnegative on $\pa D$ since \eqref{a.cond} implies that
\begin{align}
  f_1|_{u_3=1-u_2} &= (1-u_2)\big(\alpha_{11}+(\gamma_{22}-\beta_{12})u_2\big)
	\ge \alpha_{11}(1-u_2)^2 \ge 0, \label{a.f11} \\
	f_1|_{u_2=0} &= \alpha_{11} + \beta_{11}(1-u_3) \ge \alpha_{11}u_3 \ge 0, 
	\label{a.f12} \\
	f_1|_{u_3=0} &= (1-u_2)\big((\alpha_{11}+\beta_{11})(1-u_2)
	+ \alpha_{22}+\gamma_{22}\big) \ge 0, \label{a.f13} \\
	f_2|_{u_1=0} &= \alpha_{22} + \gamma_{22}(1-u_2) \ge \alpha_{22}u_2 \ge 0, 
	\label{a.f21} \\
	f_2|_{u_3=1-u_1} &= (1-u_1)\big(\alpha_{22}(1-u_1) 
	+ (\alpha_{11}+\beta_{11}-\beta_{12})u_1\big) \ge \alpha_{22}(1-u_1)^2 \ge 0, 
	\label{a.f22} \\
	f_2|_{u_3=0} &= (1-u_1)\big((\alpha_{22}+\gamma_{22})(1-u_1)
	+ (\alpha_{11}+\beta_{11})u_1\big) \ge 0. \label{a.f23} 
\end{align}
Furthermore, a straightforward computation gives
$$
  \Delta_{(u_2,u_3)}f_1 = -\Delta_{(u_1,u_3)}f_2 
	= 2(\alpha_{11}-\alpha_{22}+\beta_{11}-\gamma_{22}) \quad\mbox{in }D.
$$
Consequently, either $\Delta_{(u_2,u_3)}f_1\le 0$ or $\Delta_{(u_1,u_3)}f_2\le 0$
in $D$. By the strong maximum principle, there exists $i\in\{1,2\}$ such that
$f_i>0$ in $D$ unless $f_i\equiv 0$ in $D$. This means that $(HA)_{ii}>0$ in $D$
unless $(HA)_{ii}\equiv 0$ in $D$.

To complete the claim, we show that if one of the coefficients $(HA)_{11}$ or
$(HA)_{22}$ is identically zero in $D$, then $HA$ is positive semidefinite in $D$.
Consider first the case $(HA)_{11}\equiv 0$ in $D$, i.e.\ $f_1\equiv 0$ in $D$.
Then also $f_1\equiv 0$ on $\pa D$. We deduce from \eqref{a.f11}-\eqref{a.f13}
the relations $\alpha_{11}=\beta_{11}=0$, $\alpha_{22}=-\gamma_{22}$,
and $\gamma_{22}=\beta_{12}$ and so, 
$$
  HA = \alpha_{22}\begin{pmatrix} 0 & 0 \\ 0 & 1/u_2 \end{pmatrix}.
$$
Since $\alpha_{22}\ge 0$, $HA$ is positive semidefinite. In the remaining case
$(HA)_{22}\equiv 0$ in $D$, \eqref{a.f21}-\eqref{a.f23} lead to
$$
  HA = \alpha_{11}\begin{pmatrix} 1/u_1 & 0 \\ 0 & 0 \end{pmatrix},
$$
and because of $\alpha_{11}\ge 0$, this matrix is positive semidefinite.
This shows the claim.

{\em Step 3: Sign of the determinant of $HA$.}
By Step 2, we can assume that one of the two coefficients $(HA)_{11}$ or 
$(HA)_{22}$ is positive in $D$. We show that $\det A\ge 0$ in $D$. Then
$\det(HA)=\det H\det A\ge 0$ in $D$, and by Sylvester's criterion,
these properties give the positive semidefiniteness of $HA$. This proves
that conditions \eqref{a.cond} are sufficient for the positive semidefiniteness
of $HA$.

We consider first $\det A$ on $\pa D$. Taking into account conditions \eqref{a.cond},
we find that
\begin{align*}
  \det A(0,u_2) &= (\alpha_{22}+\gamma_{22}u_2)\big(\alpha_{11}
	+(\gamma_{22}-\beta_{12})u_2\big) 
	\ge \alpha_{22}(1-u_2)\alpha_{11}(1-u_2) \ge 0, \\
	\det A(u_1,0) &= (\alpha_{11}+\beta_{11}u_1)\big(\alpha_{22}(1-u_1)
	+ (\alpha_{11}+\beta_{11}-\beta_{12})u_1\big) \\
	&\ge \alpha_{22}(1-u_1)\alpha_{11}(1-u_1) \ge 0, \\
	\det A(u_1,1-u_1) &= \big((\alpha_{22}+\gamma_{22})(1-u_1)+\alpha_{11}+\beta_{11}
	\big) \\
	&\phantom{xx}{}\times\big(\alpha_{11}-\beta_{12}+\gamma_{22}
	+(\beta_{11}-\gamma_{22})u_1\big) \\
	&\ge (\alpha_{11}+\beta_{11})\big(-\min\{\beta_{11}-\gamma_{22},0\}
	+(\beta_{11}-\gamma_{22})u_1\big) \ge 0.
\end{align*}
We conclude that $\det A\ge 0$ on $\pa D$. 

Next, we consider the Hessian $C=D^2\det A(u)$ with respect to $u$. Since
$\det A$ is a (multivariate) quadratic polynomial in $u$, 
$C$ is a symmetric constant matrix satisfying
$$
  \det C = -\big(\beta_{11}\beta_{12}+\gamma_{22}(\alpha_{11}-\alpha_{22}-\beta_{12})
	\big)^2 \le 0.
$$
Thus, one of the two eigenvalues of $C$ is nonpositive, say $\lambda\le 0$.
Let $v\in\R^2\backslash\{0\}$ be a corresponding eigenvector, i.e.\ $Cv=\lambda v$.

Let $u\in D$ be arbitrary and let $I_u\subset\R$ be the (unique) bounded open
interval containing zero with the property that the segment $u+I_uv$ is contained
in $D$ and its extreme points belong to $\pa D$. Define $\phi(r)=\det A(u+rv)$
for $r\in I_u$. Then $\phi''(r)=v^\top Cv=\lambda|v|^2\le 0$ for all $r\in I_u$.
We infer that $\phi$ is concave and attains its minimum at the border of $I_u$.
Since $\det A\ge 0$ on $\pa D$, this implies that $\det A(u+rv)\ge 0$
for all $r\in I_u$. By choosing $r=0\in I_u$, we conclude that $\det A(u)\ge 0$.
As $u\in D$ was arbitrary, this finishes the proof.
\end{proof}

\begin{proof}[Proof of Lemma \ref{lem.HA}]
The claim \eqref{a.H2} is equivalent to the positive semidefiniteness of the
matrix $HA-\eps\Lambda$ for a suitable $\eps>0$, where
$$
  \Lambda = \begin{pmatrix} 1/u_1 & 0 \\ 0 & 1/u_2 \end{pmatrix}.
$$
Since $\Lambda=HP$, where
$$
  P = \begin{pmatrix} 1-u_1 & -u_1 \\ -u_2 & 1-u_2 \end{pmatrix},
$$
we can write $HA-\eps\Lambda=HA^\eps$ with $A^\eps=A-\eps P$. We observe that
$A^\eps$ has the same structure as $A$ with the parameters
$$
  \alpha_{11}^\eps = \alpha_{11}-\eps, \quad \alpha_{22}^\eps=\alpha_{22}-\eps,
	\quad \beta_{11}^\eps = \beta_{11}+\eps, \quad 
	\beta_{12}^\eps = \beta_{12} + \eps, 
	\quad \gamma_{22}^\eps = \gamma_{22} + \eps.
$$
>From Lemma \ref{lem.HA} we conclude that $HA^\eps$ is positive semidefinite if and
only if \eqref{a.cond} holds for the parameters 
$(\alpha_{11}^\eps,\alpha_{22}^\eps,\beta_{11}^\eps,\beta_{12}^\eps,\gamma_{22}^\eps)$ 
instead of $(\alpha_{11},\alpha_{22},\beta_{11},\beta_{12},\gamma_{22})$.
This means that $HA-\eps\Lambda$ is positive semidefinite for a suitable $\eps>0$
if and only if \eqref{1.a} hold.
\end{proof}

\begin{remark}\label{rem}\rm
Let $\alpha_{11}=\alpha_{22}=0$ but $\beta_{11}>0$ and $\gamma_{22}>0$. 
We claim that there exists $\eps>0$ such that for all $z\in\R^2$ and $u\in D$,
$$
  z^\top H(u)A(u)z \ge \eps|z|^2
$$
holds, i.e., Hypothesis H2 is satisfied for $m_i=1$, and the conclusion of
Theorem \ref{thm.ex} holds. We show that
$HA-\eps{\mathbb I}$ is positive semidefinite, where ${\mathbb I}$ is the
identity matrix in $\R^{2\times 2}$. The matrix can be written as
$$
  HA-\eps{\mathbb I} = (HA)^\eps + \frac{\eps}{1-u_1-u_2}
	\begin{pmatrix} 1 & 1 \\ 1 & 1 \end{pmatrix},
$$
where $(HA)^\eps$ has the same structure as $HA$ but with
$\beta_{11}$, $\beta_{12}$, $\gamma_{22}$ replaced by 
$\beta_{11}^\eps=\beta_{11}-\eps$, $\beta_{12}^\eps=\beta_{12}-\eps$,
$\gamma_{22}^\eps=\gamma_{22}-\eps$.
Choosing $0<\eps\le\min\{\beta_{11},\gamma_{22}\}$, conditions \eqref{a.cond}
are satisfied for these parameters. Thus, Lemma \ref{lem.HA} shows that
$(HA)^\eps$ is positive semidefinite and we conclude that also
$HA-\eps{\mathbb I}$ is positive semidefinite, proving the claim.
\qed
\end{remark}


\subsection{Verification of H3}

By definition of $f_i$, we write
$$
  f_i(u)(Dh)_i(u) = u_ig_i(u)\log u_i - u_ig_i(u)\log(1-u_1-u_2).
$$
Since $g_i(u)$ and $u_i\log u_i$ are bounded in $\overline{D}$, the first
term on the right-hand side is bounded. The second term is bounded
in $\{0<u_1+u_2\le 1-\eps\}$ by a constant which depends on $\eps$.
Moreover, we have $g_i(u)\le 0$ in $\{1-\eps<u_1+u_2<1\}$ by assumption,
which implies that $-u_ig_i(u)\log(1-u_1-u_2)\le 0$ in $\{1-\eps<u_1+u_2<1\}$.
Thus, $f_1(u)(Dh)_1(u)\le c$ for a suitable constant $c>0$.


\section{Proof of Corollary \ref{coro.ex}}\label{sec.coro}

The corollary follows from Theorem \ref{thm.Jue14} and Theorem \ref{thm.ex}
by specifying the diffusivities according to \eqref{1.SKT}.
The requirement of the symmetry of $H(u)A(u)$ leads to the conditions
$a_{11}=a_{21}$, $a_{22}=a_{12}$, and $a_{20}-a_{10} = a_{11}-a_{22}$, whereas 
\eqref{1.cond} becomes $a_{10}>0$, $a_{20}>0$, and 
$-a_{12}< a_{10}+2\min\{a_{20}-a_{10},0\}$. Taking into account that 
$a_{10}\le a_{20}$, the last condition is equivalent to $-a_{12}< a_{10}$, 
and this inequality holds since $a_{10}$ is positive.
Finally, Hypothesis H3 follows from the inequality 
$g_i(u)=b_{i0}-b_{i1}u_1-b_{i2}u_2\le b_{i0}-\min\{b_{i1},b_{i2}\}(u_1+u_2)\le 0$ 
for $1-\eps<u_1+u_2<1$, where $\eps=\min\{\eps_1,\eps_2\}$ and
$\eps_i=1-b_{i0}/\min\{b_{i1},b_{i2}\}\in(0,1)$.


\end{document}